\documentclass[12pt]{article}
\usepackage[utf8]{inputenc}
\usepackage{amsmath}
\usepackage{bbm}
\usepackage{mathtools}
\usepackage{amsfonts}
\usepackage{dsfont}
\usepackage{amssymb}
\usepackage[margin=1in,headheight=1in, asymmetric]{geometry}
\usepackage{amsthm}
\usepackage{color}

\newtheorem{theorem}{Theorem}

\newtheorem{lemma}[theorem]{Lemma}
\newtheorem{definition}[theorem]{Definition}

\newtheorem{corollary}[theorem]{Corollary}

\newtheorem{assumption}[theorem]{Assumption}

\usepackage{comment}
\usepackage{setspace}

\newcommand{\im}{\mathbf{i}}
\newcommand{\tr}{\mathrm{T}}

\title{A characterization of real matrix semigroups}
\author{Benedict Bauer\thanks{Financial support from the Austrian Science Fund (FWF) under grants P~30750
and Y~1235 is gratefully acknowledged.}\\
University of Vienna \\
\tt{benedict.bauer@univie.ac.at}\\
\\
 Stefan Gerhold \\
TU Wien \\
\tt{sgerhold@fam.tuwien.ac.at}
}

\date{\today}

\numberwithin{equation}{section}
\numberwithin{theorem}{section}

\begin{document}

\maketitle

\begin{abstract}
  We characterize all real matrix semigroups satisfying a mild boundedness
  assumption, without assuming continuity. Besides the continuous solutions
  of the semigroup functional equation,
  we give a description of solutions arising from non-measurable solutions
  of Cauchy's functional equation.
  To do so, we discuss the primary decomposition and the
  Jordan--Chevalley decomposition of a matrix semigroup.
  Our motivation stems from a characterization of all
  multi-dimensional self-similar Gaussian Markov processes,
  which is given in a companion paper.
\end{abstract}

\smallskip


\section{Introduction}

It is a classical fact that all continuous matrix-valued functions $g:[0,\infty)\to\mathbb{R}^{n\times n}$ satisfying the semigroup property
\begin{align}\label{eq:sg def}
  g(x+y)&=g(x)g(y),\quad x,y\geq 0, \\
  g(0)&=1,
\end{align}
are given by the maps $g(x)=\exp(Mx)$, where $M\in\mathbb{R}^{n\times n}$
is arbitrary. See, e.g., Problem~2.1 in~\cite{EnNa00} and the subsequent
discussion. To the best of our knowledge, few authors have considered
non-continuous solutions of~\eqref{eq:sg def} in the multidimensional
case $d>1$. Kuczma and Zajtz~\cite{KuZa66} determine all matrix semigroups $(g(x))_{x\in\mathbb R}$ where $x\mapsto g(x)$
is measurable. Zajtz~\cite{Za71} characterizes the matrix semigroups $(g(x))_{x\in\mathbb Q}$ indexed
by rational numbers. The one-dimensional case, which is equivalent to Cauchy's functional
equation,
\begin{equation}\label{eq:cauchy}
    f(x)+f(y)=f(x+y), \quad f:\mathbb{R}\to \mathbb{R},
\end{equation}
 is well studied, on the other hand (see~\cite{Ac66,BiGoTe87}).
Our motivation to investigate non-continuous matrix semigroups
stems from probability theory: In a companion paper, we develop a
 characterization of all self-similar Gaussian Markov processes.
By self-similarity, the bivariate covariance function of such a multi-dimensional stochastic process can
be transformed to a matrix function of a single argument, which must satisfy~\eqref{eq:sg def}.
See Section~\ref{se:prelim} for some more details.
Our main result (Theorem~\ref{thm:main}) determines all solutions of~\eqref{eq:sg def} which satisfy a mild boundedness assumption.
In Section~\ref{se:pd}, we provide a decomposition of the vector space
into invariant subspaces, and establish some useful properties
of the decomposition. The main proofs (in Section~\ref{se:proof})
are preceded by a discussion of the Jordan--Chevalley decomposition
of a semigroup, in Section~\ref{se:jc}.
In Appendix~\ref{se:cauchy}, we give two auxiliary results on Cauchy's functional equation.

\section{Preliminaries and main result}\label{se:prelim}

To motivate our investigation, we first recall some facts about Gaussian Markov processes (see, e.g., \cite{Li12,No12} for more information).
Let $X=(X_t)_{t\geq 0}$ be a $d$-dimensional real centered Gaussian process.
The covariance of~$X$ is a matrix-valued function $(s,t)\mapsto R(s,t)\in
\mathbb{R}^{d\times d}$ satisfying
\[
   v^\tr R(s,t) u = \mathbb{E}[X_s^\tr v\, X_t^\tr u],\quad s,t\geq0,\ u,v\in\mathbb{R}^d,
\]
and uniquely characterizes the law of the process.
Suppose that~$X$ is self-similar, and that $R(1,1)=\mathrm{id}$ is the
identity matrix. 
In terms of the covariance function, self-similarity means that
$R(as,at)=a^{2H}R(s,t)$ for $a,s,t>0$ and some self-similarity parameter $H>0$.
By a classical criterion~\cite[Theorem~V.8.1]{Do53}, $X$ is a Markov process if and only if
its covariance function satisfies
\begin{align}\label{funcequation}
    R(s,t)R(t,t)^{-1}R(t,u)=R(s,u), \quad 0\le s\le t \le u.
\end{align}
Upon introducing
$g(x):=R(e^{-x},1)$, self-similarity allows to reduce~\eqref{funcequation} to
\[
  g(x)g(y)=g(x+y),\quad x,y,\geq0.
\]
This observation has been used in dimension~$d=1$ to prove that certain
Gaussian processes do not have the Markov property (see, e.g.,
\cite[Theorem~2.3]{No12}).
Unifying and generalizing these results in a companion paper, we obtain a classification of all
$d$-dimensional self-similar Gaussian Markov processes. A full
classification requires
finding all solutions of~\eqref{eq:sg def}, without assuming continuity.
The boundedness assumption~\ref{ass:bd} stated below causes no problems, though.
To state our main result, define the rotation matrix
\[
    Q(\theta)=\begin{pmatrix}
    \cos(\theta) & \sin(\theta)\\
    -\sin(\theta) & \cos(\theta)
    \end{pmatrix}, \quad \theta\in\mathbb R,
\]
and, for even~$k$ and a function $\nu:\mathbb{R}\to\mathbb{R}$, the block-diagonal matrix
\begin{equation}\label{eq:def rot}
    Q_{k}^{\nu}(x):=\begin{pmatrix}Q(\nu(x)) & & 0\\ & \ddots &  \\0 & & Q(\nu(x))\end{pmatrix} \in \mathbb{R}^{k \times k}
\end{equation}
consisting of $k/2$ rotation matrices. In our statements, $\nu$ will denote some non-measurable (equivalently, non-continuous) solution of~\eqref{eq:cauchy}.
The following assumption is in force throughout the paper.
\begin{assumption}
In the following, $V$ denotes a real $d$-dimensional vector space
equipped with an inner product $\langle \cdot, \cdot\rangle$.
\end{assumption}
\begin{definition}
  We write $L(V)$  for the set of linear maps from~$V$ to itself. 
  The operator norm induced by the inner product 
  $\langle \cdot, \cdot\rangle$ is denoted by $\|\cdot\|_{\mathrm{op}}$.
  If the basis is clear from the context, we will identify elements
  of $L(V)$ and  $d\times d$ matrices.
\end{definition}
\begin{definition}
We say that $g:\mathbb{R}_{\geq0} \to L(V)$ is a semigroup if $g(0)=\operatorname{id}$ and
\[
g(x+y)=g(x)g(y)
\]
for all $x,y \ge 0$.
Semigroups acting on $V^{\mathbb{C}}$,  the complexification
of~$V$, are defined by the same property.
\end{definition}
\begin{definition}\label{def:elem}
We say that a semigroup $(g(x))_{x\geq0}$ in $L(V)$ is elementary if there exists an orthonormal basis such that
\[
  g(x)= \exp(Mx),\ x\geq0, \quad\text{or}\quad g(x) = Q_d^\nu(x)\exp(Mx),\ x\geq0,
\]
for some non-continuous $\nu$ satisfying Cauchy's equation~\eqref{eq:cauchy} and some matrix $M\in \mathbb{R}^{d\times d}$.
\end{definition}
\begin{assumption}\label{ass:bd}
   The semigroup $g(x)_{x\geq0}$ satisfies $\|g(x)\|_{\mathrm{op}}\leq f(x)$,
   $x\geq0$, where $f:[0,\infty)\to\mathbb{R}$ is locally bounded,
   right-continuous at~$0$ and satisfies $f(0)=1$.
\end{assumption}
We can now state our main theorem, which will be proven
at the end of Section~\ref{se:proof}.
\begin{theorem}\label{thm:main}
\label{semigroupdecomposition}
Let $(g(x))_{x\ge 0}$ be a semigroup satisfying Assumption~\ref{ass:bd}. Then there exists an orthogonal decomposition $V=\bigoplus_{i=1}^n V_i$ such that each $V_i$ is invariant under $g(x)$, and either $g(x)$ is elementary on $V_i$ or $g(x)|_{V_i} = 0$ for $x>0$.
\end{theorem}

\section{Primary decomposition of semigroups}
\label{se:pd}

\begin{definition}
Let $S$ be a linear operator on the vector space $V$. A subspace $U \subseteq V$ is called $S$-invariant if $S$ maps $U$ into $U$.
\end{definition}
%
%
The primary decomposition theorem from linear algebra \cite[Theorem 1.5.1]{OMClVi11}
 decomposes a vector space
into invariant subspaces for a given operator. 
On each subspace, the operator has a single real eigenvalue or
a pair of conjugate complex eigenvalues.
Instead of a single operator, we need the following
version for semigroups.
\begin{theorem}[Primary decomposition]
\label{decomp}
For any semigroup $(g(x))_{x \ge 0}$ of linear maps acting on $V$ there exists a decomposition $V=\bigoplus_{i=1}^n V_i$ with $\dim(V_i)\ge 1$
 such that each $V_i$ is $g(x)$-invariant for all $x\ge 0$, and for all~$i$ one of the following holds:
 \begin{itemize}
   \item[(1)] For all $x\geq0$, $g(x)$ has one eigenvalue $\lambda(x)\geq0$ on~$V_i$,
   \item[(2)]  For all $x\geq0$, $g(x)$ has eigenvalues
    $\lambda(x),\overline{\lambda}(x) \in \mathbb{C}$ on~$V_i$. They may coincide (and thus
    be real) for some values
    of~$x$, but not for all $x\geq0$.
 \end{itemize}
\end{theorem}
\begin{proof}
It is known that the primary decomposition extends to commuting sets of matrices. Indeed,
the decomposition $V=\bigoplus_{i=1}^n V_i$  into invariant subspaces follows from Theorem~5 on p.~40 in~\cite{Ja62},
applied to the span of the semigroup $(g(x))_{x \ge 0}$. Thus, we only need to argue why~(1)
or~(2) follows from the semigroup property. Assume that each $g(x)$ has only one eigenvalue $\lambda(x) \in \mathbb{R}$ on some $V_i$. Since the $g(x)$ commute, they share a common eigenvector~$v_i$, and thus we have $\lambda(x)\lambda(y)=\lambda(x+y)$ for $x,y\geq0$. Suppose $\lambda(x_0)<0$ for some $x_0>0$. Then we have $\lambda(x_0)=\lambda(x_0/2)^2$, and hence $\lambda(x_0/2)\not \in \mathbb{R}$, contradicting $\lambda(x)\in \mathbb{R}$ for all $x \ge 0$. Hence $\lambda(x)\in \mathbb{R}_{\ge 0}$ for all $x\geq0$.
\end{proof}
%

\begin{theorem}
\label{col}
Consider a semigroup  $g=(g(x))_{\ge 0}$
 acting on $V^{\mathbb{C}}$, the complexification of $V$. Then there exists a decomposition $V^{\mathbb{C}}=\bigoplus_{i=1}^n V_i$ with $\dim(V_i)\ge 1$ such that, for each~$i$ and all $x\geq0$, the space
 $V_i$ is $g(x)$-invariant and $g(x)$ has only one eigenvalue $\lambda(x) \in \mathbb{C}$ on~$V_i$.
\end{theorem}
\begin{proof}
  This is an immediate consequence of Theorem~5 on p.~40 in~\cite{Ja62} (cf.\ the preceding proof).
\end{proof}
\begin{definition}\label{def:srpd}
  We call the decomposition from Theorem~\ref{decomp} simultaneous real
primary decomposition (SRPD), omitting ``w.r.t.~$g$'' if the semigroup is clear from the context.
The component~$V_i$ is of first type in case~(1), and of second type in case~(2).
Similarly, the simultaneous primary decomposition (SPD) is the decomposition from Theorem~\ref{col}.
\end{definition}
\begin{lemma}
Let~$g$ be a semigroup acting on~$V$, and
let $V_i \subseteq V$ be a subspace of first type from the SRPD of~$V$. Then there exists a common eigenvector $v\in V$ such that $g(x)v =\lambda(x)v$ for all $x>0$.
\end{lemma}
\begin{proof}
We present an algorithm which yields the subspace of common eigenvectors. If $V_i=:V_i^{(0)}$ itself is this subspace, then we are done. Otherwise there exists $x_0>0$ such that the eigenspace $V_i^{(1)}$ of $g(x_0)$ is a strict subspace of $V_i$. For any $y>0$ and any $v \in V_i^{(1)}$ we have
\begin{align*}
0=g(y)0&=g(y)[\lambda(x_0)v-g(x_0)v]\\
 &=\lambda(x_0)[g(y)v]-g(x_0)[g(y)v].
\end{align*}
It follows that $V_i^{(1)}$ is $g(x)$-invariant for all $x>0$. Now either $V_i^{(1)}$ consists of common eigenvectors, or we can again find $x_1>0$ such that the eigenspace $V_i^{(2)}$ of $g(x_1)$ in $V_i^{(1)}$ is a strict subspace of $V_i^{(1)}$. Repeating this argument yields a sequence of nontrivial subspaces $V_i^{(n)}$ whose dimensions are strictly decreasing, hence it has to terminate. Clearly, the final vector space in this sequence is  the space of common eigenvectors of the semigroup. 
\end{proof}
\begin{lemma}
Let~$g$ be a semigroup acting on~$V$,
and let $V_i$ be a subspace from the SRPD of~$V$ such that there exists $x_0>0$ with $\lambda(x_0)=0$. Then we have $\lambda(x)=0$ for all $x>0$. Furthermore $V_i \subseteq \operatorname{ker}(g(x))$ for all $x>0.$
\end{lemma}
\begin{proof}
By the previous lemma there exists a common eigenvector $v \in V_i$. (Since commuting matrices are simultaneously triangularizable it follows that they share a common eigenvector $v \in V^\mathbb{C}$.) We have
\[
\lambda(x+y)v=g(x+y)v=g(x)g(y)v=\lambda(x)\lambda(y)v.
\]
Hence $\lambda$ satisfies $\lambda(x)\lambda(y)=\lambda(x+y)$. Since $\lambda(x_0)=0$ we have $\lambda(x_0/n)=0$ for all $n \in \mathbb{N}$ since $\lambda(x_0/n)^n=0$. Let $y >0$, then there exists $n \in \mathbb{N}$ such that $x_0/n <y$. We obtain
\[
\lambda(y)=\lambda(x_0/n)\lambda(y-x_0/n)=0.
\]
Since $\lambda(y)=0$ for all $y>0$ it follows that the characteristic polynomial of $g(y)$ satisfies $\chi_{g(y)}(Z)=Z^{\operatorname{dim}(V_i)}$
 and hence, by the Cayley--Hamilton theorem, $g(y)^{\operatorname{dim}(V_i)}\equiv 0$. For $v \in V_i$ we obtain
\[
g(y)v=g(y/{\operatorname{dim}(V_i)})^{\operatorname{dim}(V_i)}v=0. \qedhere
\]
\end{proof}
\begin{corollary}
Let~$g$ be a semigroup acting on~$V$, with SRPD
$V=\bigoplus_{i=1}^n V_i$.
For any $x>0$ we have $\operatorname{ker}(g(x))=\bigoplus_{i\in I} V_i$, where 
\[
I=\{i\, |\, V_i \text{ of type 1 with } \lambda \equiv 0\}.
\]
\end{corollary}
\begin{proof}
If $V_i \subseteq \operatorname{ker}(g(x))$, then $0=\lambda(x) \in \mathbb{R}$ for all $x \ge 0$ and hence $V_i$ is of type 1. By the previous lemma we have
\[
\operatorname{ker}(g(x)) \supseteq \bigoplus_{i\in I} V_i.
\]
If the inclusion was strict, then there would exist $V_j$ with $j \not \in I$ such that $\operatorname{ker}(g(x))\cap V_j \not = \emptyset$. But since $g(x)$ is invertible on $V_j$ this gives a contradiction, hence we have equality.
\end{proof}
\begin{corollary}
\label{zero}
Let $g$ be a semigroup acting on~$V$. Then there exists a decomposition $V=V_1 \oplus V_2$, such that $g(x)|_{V_1}$ is invertible for all $x \ge 0$ and $g(x)|_{V_2}\equiv 0$ for $x>0$.
\end{corollary}

\section{Multiplicative Jordan--Chevalley decomposition of semigroups}\label{se:jc}

Due to Corollary~\ref{zero}, from now on we assume in most of our statements
that~$g(x)$ is invertible for all $x\geq0$.
%
%
A standard result from linear algebra, the Jordan--Chevalley decomposition, asserts that any matrix $A$ can be uniquely decomposed as $A=D+N$, where $D$ is diagonalizable, $N$ is nilpotent and $D$ and $N$ commute. 
 If $A$ is invertible, then we can express it as $A=D(\operatorname{id}+D^{-1}N):=DT$ with $T$  unipotent and commuting with $D$.
\begin{definition}
For an invertible linear map $A$ on $\mathbb{K}^d$ with $\mathbb{K} \in \{\mathbb{R},\mathbb{C} \}$, the multiplicative decomposition $A=DT$ into commuting factors with $D$  diagonalizable and $T$  unipotent, is called the multiplicative Jordan--Chevalley decomposition. 
\end{definition}
For background on the (multiplicative) Jordan--Chevalley decomposition, we refer to Section~15.1 in~\cite{Hu75}. We now analyze the structure of the multiplicative Jordan--Chevalley decomposition
of a semigroup.
\begin{theorem}
Let $(g(x))_{x\ge 0}$ be a semigroup of invertible linear maps acting on $\mathbb{K}^d$ with $\mathbb{K} \in \{\mathbb{R},\mathbb{C} \}$ and let $g(x)=D(x)T(x)$ be the multiplicative Jordan--Chevalley decomposition of each $g(x)$. Then $((D(x))_{x\ge 0})$ and $(T(x))_{x\ge 0}$ each form a semigroup, and the two families commute with each other, i.e. $T(x)D(y)=D(y)T(x)$ for all $x,y \ge 0$.
\end{theorem}
\begin{proof}
We can w.l.o.g.\ assume that $\mathbb{K}=\mathbb{C}$, since by uniqueness of the Jordan--Chevalley decomposition $D(x) \in \mathbb{R}^{d \times d}$ and $T(x) \in \mathbb{R}^{d \times d}$ if $g(x) \in \mathbb{R}^{d \times d}$ for $x \ge 0$. Take the SPD from Theorem~\ref{col}, $V=\bigoplus_{i=1}^k V_i$, so that each $g(x)$ has
only one eigenvalue on $V_i$. Denote by $g(x)|_{V_i}=D_i(x)T_i(x)$ the multiplicative Jordan--Chevalley decomposition $g(x)$ restricted to $V_i$. Denote by $\lambda_i(x)$ the eigenvalue of $g(x)$ on $V_i$. Clearly, $D_i(x)=\lambda_i(x)\operatorname{id}$. Since the $(g_i(x))_{x \ge 0}$ are a commuting family of matrices, they share a common eigenvector $v_i \in V_i$. We have
\[
\lambda_i(x)\lambda_i(y)v_i=g_i(x)g_i(y)v_i=g_i(x+y)v_i=\lambda_i(x+y)v_i,
\]
and hence $(D_i(x))_{x \ge 0}$ is a semigroup. Since each $D_i(x)$ is a multiple of the identity, it commutes with every linear map and hence
\[
T_i(x)T_i(y)=\frac{1}{\lambda_i(x)\lambda_i(y)}g_i(x)g_i(y)=\frac{1}{\lambda_i(x+y)}g_i(x+y)=T_i(x+y),
\]
which shows that $(T_i(x))_{x \ge 0}$ is also a semigroup. The result then follows, since by uniqueness $(\bigoplus D_i(x))\otimes(\bigoplus T_i(x))$ is the multiplicative Jordan--Chevalley decomposition of $g(x)$.
\end{proof}
\begin{theorem}
\label{exponential}
Let $(g(x))_{x\ge 0}$ be a semigroup of invertible linear maps acting on $V$ and let $g(x)=D(x)T(x)$
be its multiplicative Jordan--Chevalley decomposition.
 Then there exist commuting real diagonalizable linear maps $J(x)$ and commuting real nilpotent linear maps $N(x)$ satisfying
\begin{center}
\begin{enumerate}
    \item $J(x)+J(y)=J(x+y)$
    \item $N(x)+N(y)=N(x+y)$
    \item $J(x)N(y)=N(y)J(x)$ for all $x,y \ge 0$
\end{enumerate}
\end{center}
such that
\[
D(x)=\exp(J(x)),\quad T(x)=\exp(N(x))
\]
and
\[
g(x)=\exp(J(x)+N(x)).
\]
\end{theorem}
\begin{proof}
Let $V=\bigoplus_{i=1}^n V_i$ be the SRPD and let $g_i(x):=g(x)|_{V_i}$. Assume first that $V_i$ is of first type and $\lambda_i(x)$ is the single positive eigenvalue of $g_i(x)$. Then the multiplicative Jordan--Chevalley decomposition on $V_i$ is $g_i(x)=\lambda_i(x)\operatorname{id}T_i(x)$. Define $J_i(x):=\log(\lambda_i(x))\operatorname{id}$. Since $T_i(x)-\operatorname{id}$ is nilpotent, we have
\[
N_i(x):=\log(T_i(x))=\sum_{k=1}^{d-1}\frac{1}{k}(\operatorname{id}-T_i(x))^k.
\]
Notice that since the logarithm converges for all unipotent matrices the exponential map between the Lie algebra of nilpotent matrices and the Lie group of unipotent matrices is bijective
(see p.~35 in~\cite{GoWa98}).
Since $T_i(x)$ is a semigroup we have
\[
\exp(N_i(x))\exp(N_i(y))=\exp(N_i(y))\exp(N_i(x))=\exp(N_i(x+y)).
\]
Rewrite this as
\[
\exp(N_i(y))^{-1}\exp(N_i(x))\exp(N_i(y))=\exp\big(\exp(N_i(y))^{-1}N_i(x)\exp(N_i(y))\big)=\exp(N_i(x)).
\]
Since $\exp(N_i(y))^{-1}N_i(x)\exp(N_i(y))$ is nilpotent and $\exp$ is bijective we have
\[
\exp(N_i(y))^{-1}N_i(x)\exp(N_i(y))=N_i(x).
\]
For any $t \in \mathbb{R}$ set $N_i^t(x):=t\operatorname{id}+N_i(x)$. $N_i^t(x)$ is invertible for $t \not = 0$. Clearly
\[
N_i^t(x)\exp(N_i(y))=\exp(N_i(y))N_i^t(x).
\]
Applying the same idea again yields
\[
\exp(N_i^t(x)^{-1}N_i(y)N_i^t(x))=\exp(N_i(y)).
\]
Again by uniqueness we obtain
\[
N_i^t(x)N_i(y)=N_i(y)N_i^t(x)
\]
or equivalently
\[
N_i(x)N_i(y)=N_i(y)N_i(x).
\]
By commutativity we obtain
\[
\exp(N_i(x)+N_i(y))=\exp(N_i(x))\exp(N_i(y))=\exp(N_i(x+y))
\]
and hence, again by uniqueness, we have
$N_i(x)+N_i(y)=N_i(x+y)$. It follows that $J_i(x)$ and $N_i(x)$ satisfy the required conditions.

In the second case $V_i$ is of second type. By Theorem~\ref{col} and since $g_i(x)$ is real, $V_i$ decomposes over $\mathbb{C}$ as $V_i=U_1 \oplus U_2$ where $g_i(x)$ has only one eigenvalue on each $U_j$ and $U_1$ is isomorphic to $U_2$ with isomorphism given by $u \in U_1 \mapsto \overline{u} \in U_2$. By taking the principal branch of the logarithm, in the same manner as in the real case we obtain commuting $J, N$ on $U_1$ and $\overline{J}, \overline{N}$ on $U_2$ satisfying Cauchy's equation such that
\[
g_i(x)=\exp(J(x)+N(x)) \oplus \exp(\overline{J}(x)+\overline{N}(x)).
\]
Notice that $J(x)=(\mu(x)+\im\nu(x))\operatorname{id}$, where
$\im=\sqrt{-1}$ and $\nu(x)$ satisfies Cauchy's functional equation on $\mathbb{R}/[-\pi,\pi)$. By Lemma~\ref{lift},
 we can lift any solution on $\mathbb{R}/[-\pi,\pi)$ to a solution on $\mathbb{R}$ such that linearity is preserved. From now on denote by $\nu(x)$ this lift.
Hence if we choose any basis of $U_1$ and its complex conjugate on $U_2$ we obtain that $g_i(x)$ is similar to
\[
g_i(x)\sim \exp\left[
    \begin{pmatrix}
    J(x)+N(x) & 0\\
    0 & \overline{J}(x)+\overline{N}(x)
    \end{pmatrix}
\right].
\]
Taking the similarity transform with the matrix
\[
A:= \frac{1}{\sqrt{2}}
\begin{pmatrix}
\operatorname{id} & \im\cdot \operatorname{id}\\
\im \cdot \operatorname{id} & \operatorname{id}
\end{pmatrix},
\]
where $\operatorname{id}=\operatorname{id}_{\operatorname{dim}(U_1)}$,
we obtain
\begin{multline*}
A\begin{pmatrix}
    J(x)+N(x) & 0\\
    0 & \overline{J}(x)+\overline{N}(x)
    \end{pmatrix}
A^{-1}
\\
=\begin{pmatrix}
\operatorname{Re}(J(x)+N(x)) & \operatorname{Im}(J(x)+N(x))\\
-\operatorname{Im}(J(x)+N(x)) & \operatorname{Re}(J(x)+N(x))
\end{pmatrix} \in L(V_i).
\end{multline*}
Since matrix similarity over $\mathbb{C}$ is equivalent to matrix similarity over $\mathbb{R}$ for two real matrices, there exists a real matrix $B_i$ on $V_i$ such that
\begin{align}
\label{char}
g_i(x)=\exp\left[
B_i\begin{pmatrix}
\operatorname{Re}(J(x)) & \operatorname{Im}(J(x))\\
-\operatorname{Im}(J(x)) & \operatorname{Re}(J(x))
\end{pmatrix}B_i^{-1}+
B_i\begin{pmatrix}
\operatorname{Re}(N(x)) & \operatorname{Im}(N(x))\\
-\operatorname{Im}(N(x)) & \operatorname{Re}(N(x))
\end{pmatrix}B_i^{-1}
\right].
\end{align}
Setting
\begin{align*}
J_i(x):=B_i\begin{pmatrix}
\operatorname{Re}(J(x)) & \operatorname{Im}(J(x))\\
-\operatorname{Im}(J(x)) & \operatorname{Re}(J(x))
\end{pmatrix}B_i^{-1}
=B_i\begin{pmatrix}
\operatorname{id}\mu(x) & \operatorname{id}\nu(x)\\
-\operatorname{id}\nu(x) & \operatorname{id}\mu(x)
\end{pmatrix}B_i^{-1}
\end{align*}
and
\[
N_i(x):=B_i\begin{pmatrix}
\operatorname{Re}(N(x)) & \operatorname{Im}(N(x))\\
-\operatorname{Im}(N(x)) & \operatorname{Re}(N(x))
\end{pmatrix}B_i^{-1},
\]
we see that $J_i$ and $N_i$ satisfy the desired conditions, with $D_i(x)=\exp(J_i(x))$ and $T_i(x)=\exp(N_i(x))$ by uniqueness of the Jordan--Chevalley decomposition. The direct sums $\bigoplus_{i=1}^n J_i(x)$ and $\bigoplus_{i=1}^n N_i(x)$ give the required matrices. Furthermore in the case where there exists $x>0$ for which $g_i(x)$ has a complex eigenvalue $\lambda(x)=\mu(x)+\im\nu(x)$ we have
\[
g_i(x)=B_i\underbrace{\begin{pmatrix}
\cos(\nu(x))\operatorname{id} & \sin(\nu(x))\operatorname{id}\\
-\sin(\nu(x))\operatorname{id} & \cos(\nu(x))\operatorname{id}
\end{pmatrix}}_{U_i(x)}B_i^{-1}\exp(\mu(x)\operatorname{id}+B_iN(x)B_i^{-1}),
\]
where $U_i(x) \in SO(\operatorname{dim}(V_i))$.
Recall the matrix defined in~\eqref{eq:def rot}.
By changing the order of the basis, we have that $U_i$ is similar to the block diagonal matrix
(recall the notation~\eqref{eq:def rot})
\begin{align}
\label{rep1}
U_i(x) \sim
\begin{pmatrix}Q(\nu(x)) & & \\ & \ddots & \\ & & Q(\nu(x))\end{pmatrix}
=Q_{\operatorname{dim}(V_i)}^{\nu}(x).
\end{align}
Hence
\begin{equation}
\label{secondtype}
g_i(x)=\tilde{B}_iQ_{\operatorname{dim}(V_i)}^{\nu}(x)\tilde{B}_i^{-1}\exp(\mu(x)\operatorname{id}+B_iN(x)B_i^{-1}),
\end{equation}
where $\tilde{B}_i$ is the composition of $B_i$ with some permutation matrix $P$.
\end{proof}

\section{Proof of the main result}\label{se:proof}

After providing some final preparatory results, this section
ends with the proof of Theorem~\ref{semigroupdecomposition}.
Consider again the SRPD from Theorem~\ref{decomp}. Now on each $V_i$, $g(x)$ has either one positive eigenvalue $\lambda_i(x)=e^{\mu_i(x)}$ or two complex conjugate eigenvalues $\lambda_i(x)=e^{\mu_i(x)+\im\nu_i(x)}$ and $\overline{\lambda_i}(x)=e^{\mu_i(x)-\im\nu_i(x)}$, where it is possible
that $\nu_i(x)=\pi$ for some values of~$x$. If $V_i$ is of first type, set $\nu_i(x)=0$. Each $\nu_i$ is a solution to Cauchy's functional equation. Consider then the set $\{f | f=\nu_i \text{ or }f=-\nu_i \text{ for some } 1\le i \le n \}$ and partition it into equivalent solutions,
according to Definition~\ref{def:c eq}. This then gives a partition of the index set $\{1,\dots,n \}=\bigcup_{l=1}^k I_l$ in the following manner: If $\nu_i\sim \nu_j$ or $\nu_i\sim -\nu_j$, then $i,j$ are in the same subset of the partition. This is well-defined, since if $f\sim g$ then $-f\sim -g$.
Set $W_l:=\bigoplus_{i\in I_l}V_i$. 
\begin{definition}\label{def:psrpd}
  We call the decomposition $V=\bigoplus_{i=1}^k W_i$  the partitioned SRPD of~$V$.
\end{definition}
Furthermore associate with each $W_j$ one solution $\eta_j$ such that $\eta_j \sim \nu_i$ with $i \in I_j$. If $\eta_j$ is linear we always take $\eta_j \equiv 0$. Notice that for $i \neq j$ $\eta_i \not \sim \eta_j$, hence there can be at most one $W_i$ with $\eta_i=0$ and furthermore this is the only $W_i$ which can have odd dimension since it contains all $V_j$ of type~1 (recall Definition~\ref{def:srpd}).
\begin{theorem}
\label{col2}
Let $(g(x))_{x\ge 0}$ be a semigroup of invertible linear maps acting on $V$ and denote by $V=\bigoplus_{i=1}^k W_i$ its partitioned SRPD. Denote by $\eta_i$ the solution associated with~$W_i$.
 If $\eta_i$ is non-continuous then there exists a change of basis $A_i$ on $W_i$ such that
\[
g(x)|_{W_i}=A_iQ_{d_i}^{\eta_i}(x)A_i^{-1}\exp(G_i(x))
\]
where $d_i=\operatorname{dim}(W_i)$ and
\[
A_iQ_{d_i}^{\eta_i}(x)A_i^{-1}G_i(y)=G_i(y)A_iQ_{d_i}^{\nu_i}(x)A_i^{-1}
\]
for all $x,y\ge 0$.
If $\eta_i=0$ then
\[
g(x)|_{W_i}=\exp(G_i(x)).
\]
In both cases $G_i(x):W_i \to W_i$ is a commuting family of matrices on $W_i$ satisfying Cauchy's functional equation. Furthermore if $v$ is a common eigenvector of $(G_i(x))_{x\ge 0}$ such that $G_i(x)v=[\mu(x)+\im\nu(x)]v$, then $\nu$ is linear in $x$.
\end{theorem}
\begin{proof}
Assume first that $\eta_i$, the associated solution of Cauchy's functional equation, is non-continuous. For each $W_i$ we have the decomposition $W_i=\bigoplus_{j \in I_i} V_j$ where $V_j$ are subspaces from the
SRPD. Since $\eta_i$ is non-continuous, each $V_j$ in the direct sum has to be of second type, hence by \eqref{secondtype} on each $V_j$ we have
\[
g_j(x)=\tilde{B}_jQ_{\operatorname{dim}(V_j)}^{\nu_j}(x)\tilde{B}_j^{-1}\exp\big(\mu(x)\operatorname{id}+B_jN(x)B_j^{-1}\big).
\]
By definition of $\eta_i$ we have $\eta_i \sim \nu_j$ and hence there exists $c_j \in \mathbb{R}$ such that $\nu_j(x)=\eta_i(x)+c_j x$. Let
\[
C_j:=c_j\begin{pmatrix}L & & \\ & \ddots & \\ & & L\end{pmatrix}
\]
with
\[
L:=\begin{pmatrix}
0 & 1\\
-1 & 0
\end{pmatrix},
\]
where $C_j$ is of dimension $\operatorname{dim}(V_j)\times \operatorname{dim}(V_j)$. Then we have
\begin{align*}
g_j(x)&=\tilde{B}_jQ_{\operatorname{dim}(V_j)}^{\nu_j}(x)\tilde{B}_j^{-1}\exp\big(\mu_j(x)\operatorname{id}+B_jN_j(x)B_j^{-1}\big)\\
&=\tilde{B}_jQ_{\operatorname{dim}(V_j)}^{\eta_i}(x)\exp(C_jx)\tilde{B}_j^{-1}\exp\big(\mu_j(x)\operatorname{id}+B_jN_j(x)B_j^{-1}\big)\\
&=\tilde{B}_jQ_{\operatorname{dim}(V_j)}^{\eta_i}(x)\tilde{B}_j^{-1}\exp\big(\tilde{B}_jC_jx\tilde{B}_j^{-1}+\mu_j(x)\operatorname{id}+B_jN_j(x)B_j^{-1}\big).
\end{align*}
Set $\tilde{G}_j(x):=\tilde{B}_jC_jx\tilde{B}_j^{-1}+\mu_j(x)\operatorname{id}+B_jN_j(x)B_j^{-1}$. Then
\begin{equation}\label{eq:Gi}
A_i:=\bigoplus_{j \in I_i} \tilde{B}_j \quad \text{and} \quad G_i(x):=\bigoplus_{j \in I_i} \tilde{G}_j(x)
\end{equation}
satisfy
\[
g(x)|_{W_i}=\bigoplus_{j\in I_i} g_j(x)=A_iQ_{d_i}^{\eta_i}(x)A_i^{-1}\exp(G_i(x)).
\]
By construction the common eigenvalues of $\exp(G_i(x))$ satisfy $\lambda(x)=e^{\mu_j(x)\pm\im c_jx}$. 
If $\eta_i \equiv 0$, then by Theorem~\ref{exponential} there exists $G_i(x)$ such that $g(x)|_{W_i}=\exp(G_i(x))$. By the definition of $W_i$, it follows that the imaginary parts of the eigenvalues of $G_i(x)$ have to be equivalent to $\eta_i\equiv 0$, hence they have to be linear.
\end{proof}
\begin{theorem}
\label{thm2}
Let $g$ be a non-degenerate semigroup acting on~$V$ which satisfies 
Assumption~\ref{ass:bd}. Then there exists a matrix $M$ and a semigroup $S(x) \in SO(d)$, where $SO(d)$ is the set of special orthogonal matrices, such that
\[
g(x)=S(x)\exp(Mx), \quad x\geq0,
\]
with $M$ commuting with $S(x)$.
\end{theorem}
\begin{proof}
Consider the SRPD $V=\bigoplus_{i=1}^n V_i$ from Theorem~\ref{decomp}, and define $g_i(x):= g(x)|_{V_i}$. Clearly,  $\|g_i(x)\|_{\operatorname{op}}\le \|g(x)\|_{\operatorname{op}}\le f(x)$. For each eigenvalue $\lambda_i(x)=e^{\mu_i(x)+\im \nu_i(x)}$ on $V_i$ we have
\[
|\lambda_i(x)| \le \|g_i(x)\|_{\operatorname{op}}\le f(x).
\]
Since $\lambda_i(x)\lambda_i(y)=\lambda_i(x+y)$, we have $|\lambda_i(x)| |\lambda_i(y)|=|\lambda_i(x+y)|$. Moreover, $\mu_i(x)=\log(|\lambda_i(x)|)$, and so $\mu_i(x) \le \log(f(x))$ for $x\ge 0$ and $\mu_i(x)+ \mu_i(y)=\mu_i(x+y)$. It follows that $\mu_i$ is locally bounded and hence that $\mu_i(x)=a_ix$ for some $a_i \ge 0$. As in the proof of Theorem~\ref{exponential}, on $V_i$ we have $g_i(x)=\exp(J_i(x))T_i(x)$ where $J_i(x)$ is diagonalizable with eigenvalues $a_ix\pm\im\nu_i(x)$, and thus
\[
\|T_i(x)\|_{\operatorname{op}}\le f(x)e^{-a_ix}.
\]
As in Theorem~\ref{exponential}, set
\[
N_i(x):=\log(T_i(x))=\sum_{k=1}^{d-1}\frac{1}{k}(\operatorname{id}-T_i(x))^k.
\]
Then, since $\|\operatorname{id}-T_i(x)\|_{\operatorname{op}}\le 1+f(x)e^{-a_ix}$, we obtain
\begin{align}
\|N_i(x)\|_{\operatorname{op}}\le \sum_{k=1}^{d-1}\frac{1}{k}\|\operatorname{id}-T_i(x)\|_{\operatorname{op}}^k \le \sum_{k=1}^{d-1}\frac{1}{k}\big(1+f(x)e^{-a_ix}\big)^k=:F(x),
\end{align}
where $F(x)$ is again locally bounded. Since the operator norm in finite dimensions is equivalent to the $L^\infty$-norm, it follows that each entry of $N_i(x)$ is locally bounded and satisfies Cauchy's functional equation. Thus, there exists a nilpotent linear map $P_i$ such that $N_i(x)=P_ix$.
Let $V=\bigoplus_{i=1}^k W_i$ be the partitioned SRPD of $V$. Assume first that the solution
associated with~$W_i$ is $\eta_i\equiv 0$. Since $W_i = \bigoplus_{j \in I_i}V_j$ we have $g(x)|_{W_i}=\bigoplus_{j \in I_i}\exp(J_j(x)+P_jx)$. The real part of the eigenvalues of $J_j(x)$ is linear in $x$ and by Theorem~\ref{col2} also the complex parts have to be linear since $\nu_i \equiv 0$. Since $J_j(x)$ is diagonalizable and all its eigenvalues are continuous in $x$, $J_j(x)$ is continuous in $x$ and hence $J_j(x)=\tilde{M}_jx$. Setting $M_i:=\bigoplus_{j \in I_i} (\tilde{M}_j+P_j)$ and $S_i(x)= \operatorname{id}$ yields $g(x)|_{W_i}=S_i(x)\exp(M_ix)$. By Theorem~\ref{col2} and~\eqref{eq:Gi},
we have for $W_i$ with $\eta_i$ non-continuous
\[
g(x)|_{W_i}=A_iQ_{d_i}^{\eta_i}(x)A_i^{-1}\exp(G_i(x))
\]
with
\begin{align*}
G_i(x)&=\bigoplus_{j \in I_i}\big(\tilde{B}_jC_jx\tilde{B}_j^{-1}+\mu_j(x)\operatorname{id}+B_jN_j(x)B_j^{-1}\big) \\
&=\bigoplus_{j \in I_i}\big(\tilde{B}_jC_j\tilde{B}_j^{-1}+a_j\operatorname{id}+B_jP_jB_j^{-1}\big)x.
\end{align*}
Hence, setting $M_i:=\bigoplus_{j \in I_i}(\tilde{B}_jC_j\tilde{B}_j^{-1}+a_j\operatorname{id}+B_jP_jB_j^{-1})$, we have $G_i(x)=M_ix$. Thus
\[
g(x)|_{W_i}=A_iQ_{d_i}^{\eta_i}(x)A_i^{-1}\exp(M_ix).
\]
Next we show that $S_i(x):=A_iQ_{d_i}^{\eta_i}(x)A_i^{-1}$ is an isometry on $W_i$. 
 Since $\|g(x)\|_{\operatorname{op}}\le f(x)$, we have
\begin{align}
\label{positivity}
\langle v,v \rangle-\frac{1}{f(x)^2}\langle g(x)v,g(x)v\rangle \ge 0
\end{align}
for any $v \in V$. Hence, for $v \in W_i$ we have
\[
\langle v,v\rangle-\frac{1}{f(x)^2}\big\langle \exp(M_ix)S_i(x)v,\exp(M_ix)S_i(x)v\big\rangle \ge 0.
\]
Fix an arbitrary  $x_0 \ge 0$, and set $\theta_0:=\eta_i(x_0)$.
The graph of $\eta_i$ is dense in $\mathbb{R}\times \mathbb{R}$ (see Appendix~\ref{se:cauchy}), and so there is 
a sequence of positive reals $x_n$ with $\lim_{n \to \infty} x_n = 0$ such that
$\lim \eta_i(x_n)=\theta_0$.
 Then, for $v \in W_i$ we obtain
\begin{align*}
0 &\le \lim \Big( \langle v,v\rangle-\frac{1}{f(x_n)^2}\big\langle \exp(M_ix_n)S_i(x_n)v,\exp(M_ix_n)S_i(x_n)v\big\rangle\Big)\\
 &= \langle v,v\rangle-\langle S_i(x_0)v,S_i(x_0)v\rangle,
\end{align*}
where $\lim \frac{1}{f(x_n)}=1$ since $f$ is right-continuous at~$0$.
Choose $(\tilde{x}_n)$ such that $\lim \tilde{x}_n =0$ and $\lim \eta_i(\tilde{x}_n)= -\theta_0$. Then, since $S(x_0)v \in W_i$, we obtain
\begin{align*}
0 &\le \lim\Big(\langle S(x_0)v,S(x_0)v \rangle - \frac{1}{f(\tilde{x}_n)^2}\big\langle \exp(M_i\tilde{x}_n)S_i(\tilde{x}_n)S_i(x_0)v,\exp(M_i\tilde{x}_n)S_i(\tilde{x}_n)S_i(x_0)v\big\rangle\Big)\\ &= \langle S_i(x_0)v,S_i(x_0)v \rangle - \langle v,v\rangle,
\end{align*}
since $Q_{d_i}(-\theta_0)Q_{d_i}(\theta_0)=\operatorname{id}$. Hence we obtain
\[
\langle v,v\rangle =\langle S(x_0)v,S(x_0)v\rangle,
\]
which implies that $S(x_0)$ is an isometry on $W_i$. Next we show that all $W_i$ are pairwise orthogonal. Let $v \in W_i$ and $u \in W_j$ for $i \neq j$. By Lemma~\ref{sequence}, there exists w.l.o.g.\ a sequence $\hat{x}_n$
 such that $\lim \hat{x}_n =0$, $\lim \eta_i(\hat{x}_n)=\pi$ and $\lim \eta_j(\hat{x}_n) = \theta \in (-\pi,\pi)$. Hence $\lim S(\hat{x}_n)|_{W_i}=-\operatorname{id}$ and $\lim S(\hat{x}_n)|_{W_j}=A_jQ_{d_j}(\theta)A_j^{-1}$ with $\theta \in (-\pi,\pi)$. We have the identity $Q_{d_j}(\phi)+Q_{d_j}(-\phi)=2\cos(\phi)\operatorname{id}$. Applying this to $\phi:=\theta/2$ we obtain
\[
\operatorname{id}+Q_{d_j}(\theta)=2\cos(\theta/2)Q_{d_j}(\theta/2).
\]
Hence we have $(\operatorname{id}+Q_{d_j}(\theta))^{-1}=(2\cos(\theta/2))^{-1}Q_{d_j}(-\theta/2)$. Set $\tilde{u}:=A_j^{-1}(\operatorname{id}+Q_{d_j}(\theta))^{-1}A_ju$. Then $\tilde{u}\in W_j$, and applying \eqref{positivity} to $v+\tilde{u}$ and taking the limit along $\hat{x}_n$ yields
\begin{align*}
    0 &\le \langle v+\tilde{u},v+\tilde{u}\rangle-\langle \lim S(\hat{x}_n)(v+\tilde{u}),\lim S(\hat{x}_n)(v+\tilde{u})\rangle\\
    &=\|v\|^2+\|\tilde{u}\|^2-\|\lim S(\hat{x}_n)v\|^2-\|\lim S(\hat{x}_n)\tilde{u}\|^2+2\langle v,\tilde{u}\rangle-2\langle \lim S(\hat{x}_n)v,\lim S(\hat{x}_n)\tilde{u}\rangle\\ 
    &= 2\langle v,\tilde{u}\rangle-2\langle \lim S(\hat{x}_n)v,\lim S(\hat{x}_n)\tilde{u}\rangle,
\end{align*}
where the last equality follows from the fact that each $S(\hat{x}_n)$ is an isometry on $W_i$ and $W_j$. Since the inequality above also holds for $v-\tilde{u}$, we obtain the equality
\begin{align*}
0=&\langle v, \tilde{u} \rangle - \langle\lim S(\hat{x}_n)v,\lim S(\hat{x}_n)\tilde{u}\rangle\\ =& \langle v, \tilde{u} \rangle + \langle v,A_jQ_{d_j}(\theta)A_j^{-1}\tilde{u}\rangle = \langle v, A_j(\operatorname{id}+Q_{d_j}(\theta))A_j^{-1}\tilde{u}\rangle = \langle v,u \rangle.
\end{align*}
Since $v,u$ were arbitrary, this shows orthogonality.
Hence the decomposition $V=\bigoplus_{i=1}^k W_i$ is orthogonal, and since $S_i(x)$ is an isometry on $W_i$, it follows that $S(x):=\bigoplus_{i=1}^k S_i(x)$ is in $SO(d)$. Setting $M=\bigoplus_{i=1}^k M_i$, we obtain
\[
g(x)=S(x)\exp(Mx).\qedhere
\]
\end{proof}
\begin{corollary}
\label{colelementary}
Assume that $S(x) =A Q_d^\nu(x) A^{-1}$ is a semigroup with $A$ being an invertible matrix such that $S(x)$ is an isometry for each $x \ge 0$. Then there exists an orthogonal matrix $U$ such that
\[
S(x)=U Q_{d}^{\nu}(x)U^\tr.
\]
\end{corollary}
\begin{proof}
One can easily verify the identities $S(x)+S(x)^\tr=2\cos(\nu(x))\operatorname{id}$
and $\sin(\nu(y))S(x)-\sin(\nu(x))S(y)=\sin(\nu(x)-\nu(y))\operatorname{id}$. 
Choose $x>0$ such that $\sin(\nu(x))\neq 0$. Such an $x$ clearly exists since $\nu$ is linear on $\mathbb{Q}x$. Choose any $v \in V$ of unit length and set
\[
u := \frac{S(x)v-\cos(\nu(x))v}{\sin(\nu(x))}.
\]
The definition of $u$ is invariant under the choice of~$x$ as long as $\sin(\nu(x))\neq 0$. This can be seen by noticing that
\begin{align*}
\sin(\nu(y))S(x)-\sin(\nu(x))S(y)&=\sin(\nu(x)-\nu(y))\operatorname{id} \\ \Longleftrightarrow \quad \frac{S(x)v-\cos(\nu(x))v}{\sin(\nu(x))}&=\frac{S(y)v-\cos(\nu(y))v}{\sin(\nu(y))}.
\end{align*}
Hence we obtain $S(x)v=\cos(\nu(x))v+\sin(\nu(x))u$. We have
\[
\langle u,v \rangle = \frac{1}{\sin(\nu(x))}\langle S(x)v-\cos(\nu(x))v,v \rangle=\frac{1}{\sin(\nu(x))}\big(\langle S(x)v,v\rangle-\cos(\nu(x))\big)=0,
\]
where the last equality follows from
 \[
 2\langle S(x)v,v\rangle = \langle S(x)v,v\rangle+\langle S(x)^\tr v,v\rangle=\langle S(x)v+S(x)^\tr v,v\rangle=2\cos(\nu(x)).
 \]
 Similarly we can show that $u$ is also of unit length. 
Set $H:=\operatorname{Span}(u,v)$.  Clearly, $H$ is invariant under $(S(x))_x$ and hence so is $H^\perp$ since each $S(x)\in SO(d)$. In this manner we can construct an orthonormal basis,  and we denote by $U$ the matrix associated with this change of basis. Then we have
\[
Ug(x)U^\tr =UAS(x)A^{-1}U^\tr\exp(UMU^\tr x)=Q_d^\nu(x)\exp(UMU^\tr x). \qedhere
\]
\end{proof}
\begin{lemma}
\label{limit}
Suppose that the semigroup~$g$, acting on~$V$, satisfies Assumption~\ref{ass:bd}.
Let $V=V_1 \oplus V_2$ be the decomposition of Corollary~\ref{zero} such that $g|_{V_1}$ is invertible. Then, for any $v \in V_1$ we have
\[
\lim_{x \to 0}\frac{\|g(x)v\|}{\|v\|}=1.
\]
\end{lemma}
\begin{proof}
By Theorem~\ref{thm2}, we have $g(x)|_{V_1}=S(x)\exp(Mx)$.
 Since $S(x) \in SO(V_1)$, we obtain
\[
\langle g(x)v,g(x)v \rangle = \langle \exp(Mx)v,\exp(Mx)\rangle, \quad v\in V_1,
\]
which is continuous in $x$. Hence
\[
\lim_{x\to 0}\|g(x)v\|=\lim_{x\to 0}\|\exp(Mx)v\|=\|v\|. \qedhere
\]
\end{proof}
\begin{corollary}
\label{orthogonal}
Suppose that the semigroup~$g$, acting on~$V$, satisfies Assumption~\ref{ass:bd}. Then the decomposition from Corollary~\ref{zero} is orthogonal.
\end{corollary}
\begin{proof}
Let $V=V_1 \oplus V_2$ with $g(x)|_{V_1}$ being non-degenerate and $g(x)|_{V_2}\equiv 0$. Assume $V_1$ is not orthogonal to $V_2$. Then there exists $v \in V_1$ such that $p_{V_2}(v)\neq 0$, where $p_{V_2}$ denotes the orthogonal projection onto~$V_2$. By Lemma~\ref{limit}, we have
\[
\lim_{x \to 0}\frac{\|g(x)v\|}{\|v\|}=1. 
\]
Calculating the same limit for $v-p_{V_2}(v)$ we obtain
\begin{align*}
    \lim_{x \to 0}\frac{\|g(x)(v-p_{V_2}(v))\|}{\|v-p_{V_2}(v)\|}=\lim_{x \to 0}\frac{\|g(x)v\|}{\|v\|}\frac{\|v\|}{\|v-p_{V_2}(v)\|}=\frac{\|v\|}{\|v-p_{V_2}(v)\|}.
\end{align*}
Since $\|v\|^2=\|v-p_{V_2}(v)\|^2+\|p_{V_2}(v)\|^2$ and $\|p_{V_2}(v)\|\neq 0$, we have $\frac{\|v\|}{\|v-p_{V_2}(v)\|}>1$. Hence
\[
\lim_{x \to 0}\frac{\|g(x)(v-p_{V_2}(v))\|}{\|v-p_{V_2}(v)\|}>1,
\]
but this contradicts $\|g(x)\|_{\operatorname{op}}\le f(x)$. Hence $V_1 \perp V_2$.
\end{proof}
\begin{proof}[Proof of Theorem~\ref{semigroupdecomposition}]
By Corollary~\ref{orthogonal}, we have the orthogonal decomposition $V=\tilde{V}_1 \oplus \tilde{V}_2$ 
with $g(x)|_{\tilde{V}_2}\equiv 0$ and $g(x)|_{\tilde{V}_1}$  non-degenerate. Applying Theorem~\ref{thm2} to $g(x)|_{\tilde{V}_1}$ yields the result.
\end{proof}

\appendix

\section{Cauchy's functional equation}\label{se:cauchy}

It is classical that all continuous solutions of the equation~\eqref{eq:cauchy},
 $f(x)+f(y)=f(x+y)$, are linear, and
that the non-linear solutions are not continuous, even not Lebesgue measurable, and have
dense graphs. For this,
and further references, we refer to~\cite[Section~1.1]{BiGoTe87}.
In this section, we provide two auxiliary results on Cauchy's equation.
They concern lifting solutions from an interval to the real line, resp.\ the joint behavior of two solutions
that differ by a non-linear function.
\begin{lemma}\label{lift}
Let $f:\mathbb{R} \to \mathbb{R}/[-a,a)$ be a solution to Cauchy's functional 
equation~\eqref{eq:cauchy} on $\mathbb{R}/[-a,a)$ with $a>0$. Then there exists a solution $\tilde{f}:\mathbb{R}\to \mathbb{R}$ of Cauchy's functional equation such that $f(x)\equiv \tilde{f}(x) \mod [-a,a)$. The solution~$\tilde{f}$ is linear
if and only if~$f$ is linear.
\end{lemma}
\begin{proof}
Take a Hamel basis $(r_i)_{i \in I}$ of $\mathbb{R}$ such that $r_i \in [-a,a)$ for every $i \in I$. This is clearly possible by rescaling every basis element if necessary. For any $x \in \mathbb{R}$ there exists a finite subset $I_x \subset I$ and $c_i \in \mathbb{Q}$ such that $x=\sum_{i \in I_x}c_ir_i$. Set $\tilde{f}(x)=\sum_{i \in I_x}c_if(r_i)$, then clearly
 $f(x)\equiv \tilde{f}(x) \mod [-a,a)$.
If $f$ is linear then clearly $\tilde{f}$ is linear as well. If $f$ is not linear then there exist two basis elements $r_1$ and $r_2$ such that $f(r_1)-\frac{r_1}{r_2}f(r_2)\neq 0$,
 and hence $\tilde{f}$ is also not linear.
\end{proof}
\begin{definition}\label{def:c eq}
We say that two solutions $\nu$ and $\eta$ of Cauchy's functional equation are equivalent if $\nu-\eta$ is linear.
\end{definition}
\begin{lemma}
\label{sequence}
Let $f,g:\mathbb{R}\to \mathbb{R}$ be two non-equivalent solutions of Cauchy's functional equation. Then there exists a sequence $(x_n)_{n\ge 0}$ in $\mathbb{R}_{\ge 0}$, converging to~$0$, such that either $\lim_n f(x_n)=\pi$ and $\lim_n g(x_n)=\theta$ with $|\theta|<\pi$ or vice versa.
\end{lemma}
\begin{proof}
Choose $x,y \in \mathbb{R}_{>0}$ such that the two vectors $v_1=(x,f(x),g(x))$ and $v_2=(y,f(y),g(y))$ are linearly independent and $\frac{f(x)}{x}-\frac{f(y)}{y}\neq \frac{g(x)}{x}- \frac{g(y)}{y}$. This is possible, since
\[
\frac{f(x)}{x}-\frac{f(y)}{y}= \frac{g(x)}{x}- \frac{g(y)}{y}
\]
for all $x,y>0 $ would imply that $f-g$ is linear. Assume w.l.o.g.\ that $\left| \frac{f(x)}{x}-\frac{f(y)}{y} \right| > \left| \frac{g(x)}{x}-\frac{g(y)}{y} \right|$.
Since $f$ and $g$ are both linear on $\mathbb{Q}x$ and $\mathbb{Q}y$ and $v_1$ and $v_2$ are linearly independent, there exist sequences $q_n \in \mathbb{Q}$ and $r_n \in \mathbb{Q}$ such that $q_n x -r_n y\ge 0$ for every $n$, $\lim q_n x -r_n y = 0$ and $\lim_n f(q_nx-r_ny)=\lim_n q_n f(x)-r_n f(y)=\pi$.
We show that $x_n=q_n x -r_n y$ has the required property.
 Clearly $\lim \frac{r_n}{q_n}=\frac{x}{y}$ and
\[
\pi=\lim_n q_n f(x)-r_n f(y) = \lim_n q_n \left(f(x)-\frac{r_n}{q_n} f(y)\right)=\left(f(x)-\frac{x}{y} f(y)\right)\lim q_n.
\]
Hence
\begin{align*}
|\lim q_n g(x)-r_n g(y)|=&\left|\left(g(x)-\frac{x}{y} g(y)\right)\lim q_n\right|\\ =&\left|\pi\left(\frac{g(x)}{x}-\frac{g(y)}{y}\right)/\left(\frac{f(x)}{x}-
 \frac{f(y)}{y}\right)\right| < \pi. \qedhere
\end{align*}
\end{proof}

\bigskip

\bibliographystyle{siam}
\bibliography{../literature}

\end{document}